\documentclass[11pt,reqno]{amsart}

\usepackage{graphicx}
\usepackage{color}
\usepackage{geometry}
\usepackage{calrsfs}
\usepackage{bbold}
\usepackage[T1]{fontenc} 
\usepackage{textcomp}
\usepackage{times}
\usepackage[scaled=0.92]{helvet} 

\swapnumbers
\newtheorem{theorem}{Theorem}[section]
\newtheorem{proposition}[theorem]{Proposition}
\newtheorem{lemma}[theorem]{Lemma}
\newtheorem{corollary}[theorem]{Corollary}
\newtheorem{conjecture}[theorem]{Conjecture}
\theoremstyle{definition}

\newtheorem{example}[theorem]{Example}
\newtheorem{examplepf}[theorem]{Example/Proof}
\newtheorem{remark}[theorem]{Remark}
\newtheorem{se}[theorem]{}

\usepackage{amsfonts}
\usepackage{amssymb}

\renewcommand{\tilde}{\widetilde}

\DeclareMathOperator{\Z}{\mathbf{Z}}

\DeclareMathOperator{\Q}{\mathbf{Q}}

\DeclareMathOperator{\F}{\mathbf{F}}
\DeclareMathOperator{\p}{\mathfrak{p}}

\DeclareMathOperator{\cO}{\mathcal{O}}
\DeclareMathOperator{\Fr}{\mathrm{Fr}}
\DeclareMathOperator{\ord}{\mathrm{ord}}

\title[Perfect powers on elliptic curves over function fields]{The perfect power problem for\\[1mm] elliptic curves over function fields}\thanks{This work was completed whilst the authors enjoyed the hospitality of the University of Warwick (special thanks to Richard Sharp for making it possible) and during a visit of the first author to the Hausdorff Institute in Bonn.} 
\author[G.\ Cornelissen]{Gunther Cornelissen}
\address{\normalfont{Mathematisch Instituut, Universiteit Utrecht, Postbus 80.010, 3508 TA Utrecht, Nederland}}
\email{g.cornelissen@uu.nl}
\author[J.\ Reynolds]{Jonathan Reynolds}
\address{\normalfont{INTO University of East Anglia, Norwich Research Park, Norwich, Norfolk, NR4 7TJ, United Kingdom}}
\email{jonathan.reynolds@uea.ac.uk}
\date{\today}
\thanks{}

\begin{document}

\begin{abstract}
We generalise the Siegel-Voloch theorem about $S$-integral points on elliptic curves as follows: let $K/\F$ denote a global function field over a finite field $\F$ of characteristic $p\geq 5$, let $S$ denote a finite set of places of $K$ and let $E/K$ denote a non-isotrivial elliptic curve over $K$ with $j$-invariant $j_E \in K^{p^s}-K^{p^{s+1}}$. Fix a non-constant function $f \in K(E)$ with a pole of order $N>0$ at the zero element of $E$. We prove that there are only finitely many rational points $P \in E(K)$ such that for any valuation outside $S$ for which $f(P)$ is negative, that valuation of $f(P)$ is divisible by some integer not dividing $p^sN$. We also present some effective bounds for certain elliptic curves over rational function fields, and indicate how a similar result can be proven over number fields, assuming the number field abc-hypothesis. 
 \end{abstract}

\maketitle

\section{Introduction}
To put our work in context, we cite a few results from the literature on perfect powers and $S$-integral points in linear recurrent sequences and on elliptic curves (the analogy arising from the fact that denominators of rational points on elliptic curves give rise to higher order recurrence sequences called ``elliptic divisibility'' sequences). 
 \begin{itemize} 
\item Peth\H{o} \cite{Petho}, and Shorey and Stewart have proven that a large class of linear recurrent sequences contain only finitely many  pure powers $>2$ up to factors from a given finite set of primes (see, e.g., Corollary 2 in \cite{SS}). 
\item Bugeaud, Mignotte and Siksek have applied the modular method to explicitly list all perfect powers in the Fibonacci sequence (see, e.g., \cite{BMS}). 
\item Lang and Mahler have shown that Siegel's theorem on integral points generalises to the statement that the set of $S$-integral points on a hyperbolic curve (e.g., an elliptic curve) is finite, for every finite set $S$ of valuations on a number field  (\cite{Siegel}, \cite{Lang}, \cite{Mahler}). 
\item In \cite{ERS}, it is proven that the set of denominators of points on an elliptic curve contains only finitely many $\ell$-th powers for \emph{fixed} $\ell>2$ (cf.\ also \cite{Jthesis} for a general number field). 
 \end{itemize}
 
 In this paper, we consider such questions over global function fields $K$ over a finite field $\F$ of characteristic $p\geq 5$ (where we say that $x \in K$ is a perfect $\ell$-th power if all its valuations are divisible by $\ell$). For a study of recurrent sequences in this setting, see, e.g., \cite{Marco} and references therein.  The analogue of Siegel's theorem was proven by Voloch (\cite{Voloch}; under the necessary assumption that the elliptic curve is not isotrivial). We are interested in strenghtening this by considering perfect powers $>2$ up to a finite set $S$ of  valuations in denominators of points on elliptic curves over $K$ (here, ``denominators'' refers to negative valuations of the coordinates of the point). Our main result generalizes the Siegel-Voloch theorem and at the same time gives a finiteness result that is uniform in the powers that can occur (up to the obvious $p$-th powers that arise from possible inseparability in the $j$-invariant of $E$): 

\begin{theorem} \label{main} 
Let $K$ be a global function field over a finite field $\F$ of characteristic $p \geq 5$ and $S$ a finite set of places of $K$. Suppose that $E$ is a non-isotrivial elliptic curve over $K$ with $j$-invariant $j_E \notin \F$, and let $p^s$ the largest $p$-power $p^s$ such that $j \in K^{p^s}$. Let $f$ denote a non-constant function in $K(E)$ with a pole of order $-\ord_O(f)>0$ at the zero point $O=O_E$ of the group $E$. Define the set 
\begin{equation} \label{buh} \mathcal{P}(E,K,S,f)_n:=\{ P \in E(K) : n \mid \nu(f(P)), \textrm{ for all } \nu \notin S \textrm{ with } \nu(f(P)) <0 \}. \end{equation}
Then \begin{equation} \label{PIPset}
\mathcal{P}(E,K,S,f):=\bigcup_{n \nmid \ord_O(f)\cdot p^s} \mathcal{P}(E,K,S,f)_n
\end{equation}
is finite. 
\end{theorem}

\begin{remark}
The result implies \emph{Voloch's analogue of Siegel's theorem} (\cite{Voloch}, 5.3), which states that if $j_E$ is non-constant, then 
the set of $S$-integer values of $f$ on $E$, defined as \begin{equation*} \label{VS} \mathcal{Q}(E,K,S,f):=\{P \in E(K) \colon \nu(f(P)) \geq 0 \mbox{ for all } \nu \notin S\} \end{equation*} 
is finite. This is implied by the above theorem by combining it with the equality  
\begin{equation*} \label{VS} \mathcal{Q}(E,K,S,f) = \bigcap_{n \geq 1} \mathcal{P}(E,K,S,f)_n.\end{equation*} 
\end{remark}

\begin{remark}
There is a corresponding statement for smooth curves of genus one (not necessarily with a $K$-rational point), that follows immediately from the theorem, but whose formulation is slightly more complicated, since all poles of the function $f$ might be in an inseparable extension of $K$: if $C$ is a curve of genus one over a global function field $k$ over $\F$ and $f \in k(C)-k$ is a non-constant function, then let $O \in C(K)$ denote a pole of $f$ in some finite extension $K/k$. Then if the $j$-invariant of the Jacobian of $C$ is in $K^{p^s}-K^{p^{s+1}}$, the set $\mathcal{P}(C,k,S,f)$ (defined as in (\ref{buh}) and (\ref{PIPset})) is finite. 

Also, replacing $f$ by $f^{-1}$, there is a corresponding result for functions which have a zero at $O$ (but then concerning $P$ for which $\nu(f(P))>0$ implies $n \mid \nu(f(P))$). 
\end{remark}

\begin{remark}
To show that \emph{the condition involving $p^s$ is necessary}, suppose $$E' \colon y^2=x^3+ax+b$$ is an elliptic curve of nonzero rank over $K$ and let $E$ be given by $$E \colon y^2=x^3+a^px+b^p$$ for some $a,b \in K$. Then $E(K)$ contains infinitely many $p$-th powers $(\tilde x^p, \tilde y^p)$ for $(\tilde x, \tilde y)$ running through the infinite set $E'(K)$. In a sense, this example is universal, in that a suitable Frobenius twist can be used to reduce to the case where $j_E \notin K^p$, cf.\ Proposition \ref{doug}. 
\end{remark}

\begin{remark} \label{formal}
To make the \emph{analogy with linear recurrent sequences}, one can apply the theorem to multiples of a fixed (infinite order) point $P$ in $E(K)$ and the coordinate function $x$ on a Weierstrass equation for $E$, for which $\mathrm{ord}_O(x)=-2$, when it says something about perfect powers in the associated elliptic divisibility sequence: assume for simplicity that $j_E \notin K^p$, and fix a place $\infty$ of $K$ such that the ring of functions $\cO$ regular outside $\infty$ is a PID. Factor $x(P)=A_P/B_P^2$ with $A_P$ and $B_P$ coprime in $\cO$. Then $\{B_{nP}\}$ is a divisibility sequence in the UFD $\cO$ in the conventional sense, and the theorem (with $S=\{ \infty \}$) says that it contains only finitely many perfect powers, in the usual meaning of the word. 

 As was observed in \cite{Marco} (Lemma 22), the structure of the formal group associated to $E(K_v)$ implies that if $\nu(x(nP))<0$, then $\nu(x(mnP))=\nu(nP)$ for all integers $m$, in stark contrast with the number field case, where $\{\nu(x(mnP))\}_{m \geq 1}$ in unbounded. This does not imply anything about large perfect powers, since it might be that the smallest $n$ for which $\nu(x(nP))$ is negative has very large $-\nu(x(nP))$; see the next remark. \end{remark}

\begin{remark} \label{dougex}
There is \emph{no absolute (i.e., not depending on the elliptic curve $E$) bound on the power} that can occur in denominators of elliptic curves over function fields. For example, consider the curve $$E \colon y^2+xy=x^3-t^{2d}$$ 
over the rational function field $K=\F_p(t)$ with $p = 1 \mbox{ mod } 4$, and let $\{B_m\}$ be the elliptic divisibility sequence over $\mathcal{O}=\F_p[t]$ generated by $P=(0, at^{d}) \in E(K)$ where $a$ is chosen so that $a^2=-1 \mbox{ mod } p$. Then \begin{equation*} B_1=B_2=B_3=1 \mbox{ and }B_4=t^{d}. \end{equation*} (This curve is taken from Theorem 1.5 in \cite{UlmerAnn}.)
\end{remark}

\begin{remark}
The \emph{requirement $p \geq 5$} arises from our method of proof because we apply the abc-conjecture to a ternary equation associated to the 2-division polynomial on a short Weierstrass form and we take field extensions of degree 2 and 3 in the proof (which could introduce inseparability if $p \leq 3$). 
\end{remark}

\begin{remark}
The \emph{requirement $j \notin \F$ is necessary}. For example, if $y^2=x^3+ax+b$ is a curve with $a,b \in \F$ and $K \supseteq \F(t,\sqrt{1+a t^4+b t^6})$ then $E(K)$ contains the points $$(\frac{1}{t^{2p^s}}, \frac{\sqrt{1+a t^4+b t^6}}{t^{3p^s}})$$ for all $s$, on which the $x$-coordinate has unbouded negative $t$-valuation. The theory of twists provides more such examples. 
\end{remark}

Here is an outline of the proof of the theorem. Throughout the proof we can enlarge the field $K$ to a separable extension and the set $S$ to a larger set of valuations. If the $j$-invariant is a $p^s$-power, then $E$ is the $p^s$-fold Frobenius twist of another curve $E'$, and we prove that we can work with $E'$ instead and assume $j_E \notin K^p$. We use a standard reduction from a general function $f$ to a coordinate function $x$ on a short Weierstrass equation. We use the method of ``Klein forms'' to show that the existence of a point in $\mathcal{P}(E,K,S,f)_n$ implies the existence of a solution  to a ternary equation of the form $X^2+Y^3=Z^{4n}$ in $S$-integers. We then use Mason's theorem (the ``abc-conjecture in function fields'') to bound $n$ unless it is divisible by $p$. We can conclude that the union in (\ref{PIPset}) needs to be taken over only finitely many $n$. Finally, we use the Siegel identities to prove that each individual $\mathcal{P}(E,K,S,f)_n$ is finite, or $j_E \in K^p$. 

In principle, the method is \emph{effective}, in that all occurring constants can be bounded above in terms of $E, K$ and $S$, but doing this abstractly in practice is rather painful, given that the proof involves recurrent enlargement of $K$ and $S$. Instead, we will present an example of effectively bounding the maximal perfect power exponent in a non-integral point in \ref{expl} and \ref{explH}. 

As another example of making the results explicit, we prove the following in section \ref{eee}, which shows what kind of bounds one can expect (i.e., linear in the degree of the discriminant of the curve):
\begin{proposition} \label{introprop}
Assume that $E$ is an elliptic curve over a rational function field $K=\F_q(t)$ with coefficients from $\F_q[t]$ such that all 2-torsion points on $E$ are $K$-rational and $j_E \notin K^p$. Assume that $P=2Q \in 2E(K)$ has associated elliptic divisibility sequence $\{B_n\}$. If $B_n \notin \F$ is a perfect $\ell$-th power of a polynomial in $t$, then we have the following bounds: 
$$ \ell \leq 4 \deg \Delta_E; \ \ \deg B_n \leq \frac{61}{2} \deg \Delta_E; \ \ n \leq \frac{732\deg \Delta_E}{12h(x(P))-h(j_E)},$$
where $h(x) = \max \{ \deg(A), \deg(B) \}$ if $x=A/B$ with $A$ and $B$ coprime in $\F_q[t]$. 
\end{proposition}

At the end of the paper, we briefly outline how the analogue of Theorem \ref{main} can be proven in number fields, assuming the number field abc-conjecture. 

\section{First reductions}

\begin{se} Let $K$ denote a global function field of genus $g$ over a finite field $\F$ of characteristic $p \geq 5$, let $M_K$ denote the set of all valuations of $K$, and let $S$ denote a finite set $S \subset M_K$. Let $\cO_{K,S}$ denote the ring of $S$-integers $$\cO_{K,S}=\{x \in K \colon \nu(x) \geq 0 \mbox{ for all } \nu \notin S\},$$ and $$\cO^*_{K,S}=\{x \in K \colon \nu(x) = 0 \mbox{ for all } \nu \notin S\}$$ the ring of $S$-units. We call two elements $a,b \in \cO_{K,S}$ \emph{coprime $S$-integers} if for all $\nu \notin S$, either $\nu(a)=0$ or $\nu(b)=0$. The finiteness of the class number $h_{K,S}$ of $\cO_{K,S}$ implies: 
\begin{lemma}
There exists a set $S'$ consisting of at most $h_{K,S}-1$ valuations such that $\cO_{K,S\cup S'}$ is a PID.  \qed
\end{lemma}

\end{se}

\begin{se}
Let $E$ denote an elliptic curve over $K$, with $j$-invariant $j_E \notin \F$. Fix a short Weierstrass equation $y^2 = x^3 + a x + b$ for $E/K$, which is possible since $p\geq 5$. Let $O=O_E$ denote the zero point of the group $E$. If $P \in E(K)$ is a rational point with $P \neq O$, write it in affine form as $P=(x(P),y(P))$. 
\end{se}

\begin{lemma} Theorem \ref{main} holds for a field $K$ and a set of valuations $S$ if it holds for a field $K' \supseteq K$ and a set $S'$ of $K'$-valuations that contains the extension of all $S$-valuations to $K'$, if $K'/K$ is a finite separable field extension. 
\end{lemma}

\begin{proof}
Under the given conditions, $\mathcal{P}(E,K,S,f)_n \subseteq \mathcal{P}(E,K',S',f)_n$ for all $n$, and separability of $K'/K$ implies that $j_E \in (K')^{p^s}-(K')^{p^{s+1}}$. 
\end{proof}

\begin{proposition} \label{doug} Theorem \ref{main} holds true for all elliptic curves over $K$ with non-constant $j$-invariant if it holds for all elliptic curves $E$ whose $j$-invariant $j_E$ satisfies $j_E \notin K^p$. 
\end{proposition} 

\begin{proof} Since $j_E \notin \F$, there exists an integer $s$ such that $j_E \in K^{p^s} - K^{p^{s+1}}$. Write $j = (j')^{p^s}$ for a uniquely determined $j' \in K$. There exists an elliptic curve $E'$ over $K$ with $j$-invariant $j_{E'}=j'$ such that $E$ is the image of $E'$ under the $p^s$-Frobenius map $$\Fr_{p^s} \colon (x,y) \mapsto (x^{p^s},y^{p^s})$$ (see, e.g., \cite{Ulmerlectures}, Lemma I.2.1). There is an embedding
$$  E(K) / \Fr_{p^s}(E')(K) \hookrightarrow \mathrm{Sel}(K,\Fr_{p^s}),$$
where the $p$-Selmer group $\mathrm{Sel}(K,\Fr_{p^s})$ is \emph{finite} $p$-group, as shown by Ulmer \cite{Ulmer} (Theorem 3.2 in loc.\ cit.\ if $s=1$ and $E$ has a rational $p$-torsion point; if $s=1$ in general by the argument at the start of section 3 of that paper, and for general $s$ by iteration).

Now suppose that $P \in \mathcal{P}(E,K,S,f)$, so $P$ satisfies that for all $K$-valuations $v \notin S$, if $\nu(f(P))<0$ then $n \mid \nu(f(P))$ for some $n$ not dividing $\ord_{O_E} f \cdot p^s$. We have to show that $P$ belongs to a finite set.

First, assume that $P \in \Fr_{p^s}(E')(K)$; then there is a (unique) $Q \in E'(K)$ such that $\Fr_{p^s}(Q)=P$. 
The given function $f \in K(E)-K$ extends to a function $$f':=f \circ \Fr_{p^s} \in K(E')-K,$$ and for any valuation $\nu \in M_K$, we have $$\nu(f(P)) = \nu(f(\Fr_{p^s}(Q)) =\nu(f'(Q)).$$
Finally, we have that $$\ord_{O_{E'}} f' = p^s \ord_{O_E} f.$$ 
Now $Q$ satisfies the same conditions as $P$, but for some $n$ not dividing $\ord_{O_E} f \cdot p^s = \ord_{O_{E'}} f',$ so 
$Q \in \mathcal{P}(E',K,S,f')$. Since we assume the theorem proved for $E'$ over $K$ ($j_{E'} \notin K^p$), this set is finite, so 
$$ \mathcal{P}(E,K,S,f) \cap \Fr_{p^s}(E')(K) \mbox{ is finite.}$$

Finally, if $P \in \mathcal{P}(E,K,S,f)$ in general, let $$\mathcal{R}:=\{R_1,\dots,R_m\} \subset E(K)$$ denote a set of coset representatives of $\Fr_{p^s}(E')(K)$ in $E(K)$. Then there exists $i \in \{1,\dots,m\}$ with $P-R_i \in \Fr_{p^s}(E')(K)$. Finally, observe that for fixed $R$, $P \mapsto f(P+R)$ is a non-constant function with the same order at $O_E$ as $f$; thus
$$ \mathcal{P}(E,K,S,f) \subseteq \bigcup_{i=1}^m \left[R_i + \left(\mathcal{P}(E,K,S,f(R_i+\cdot)) \cap \Fr_{p^s}(E')(K) \right) \right], $$
and the right hand side is a finite union of sets that are finite by the previous argument. 
\end{proof} 

\begin{proposition} Theorem \ref{main} holds true for all non-constant functions $f$ if it holds true for the coordinate function $x$ on a short Weierstrass model for the curve $E$.
\end{proposition} 

\begin{proof}  We claim that if $P \in \mathcal{P}(E,K,S,f)_n$ for some $n$ coprime to $\ord_O f$, then we also have that $P \in \mathcal{P}(E,K,S',x)_{n'}$ for some $n'>2=-\ord_O x$ and and extension $S' \supseteq S$, where $x$ is the $x$-coordinate of a Weierstrass model $y^2=x^3+ax+b$. The method of proof is taken from \cite{Jthesis}, 5.2.3 (cf.\ \cite{SilvermanAEC} IX.3.2.2 for a similar reduction in case of Siegel's theorem). 

Write $ f = (\varphi(x) + y \psi(x))/\eta(x)$ for polynomials $\varphi, \psi, \eta \in K[x]$ of respective degrees $d_1, d_2$ and $d_3$. 

First we compute the order of the pole of $f$ at $O$: since $x$ is of order $-2$ and $y$ of order $-3$, we find \begin{equation} \label{ordo} \ord_O(f) = \ord_O(\varphi(x)+y\psi(x))-\ord_O(\eta) =- \max\{2(d_1-d_3),2(d_2-d_3)+3\}. \end{equation}

Enlarge $S$ so that $a,b$ and all coefficients of these three polynomials are $S'$-integers and their leading coefficients are $S'$-units, keeping $\cO_{K,S'}$ a PID. If we write $x(P)=(A/B^2,C/B^3)$ in $S$-integers $A,B,C$ with $B$ coprime to $AC$, then we have the following two expressions for $f(P)$: 
\begin{align} f(P) & \label{rep1} = \frac{1}{B^{3+2(d_2-d_3)}} \cdot \frac{B^{3+2(d_2-d_1)} \left( B^{2d_1} \varphi(A/B^2)\right)+ C \left(B^{2 d_2} \psi(A/B^2)\right)} {B^{2d_3} \eta(A/B^2)} \\ &= \label{rep2}
 \frac{1}{B^{2(d_1-d_3)}} \cdot \frac{B^{2(d_1-d_2)-3} \left(C B^{2 d_2} \psi(A/B^2)\right)+\left(B^{2d_1} \varphi(A/B^2)\right)} {B^{2d_3} \eta(A/B^2)} 
 \end{align}
 
First, suppose that in (\ref{ordo}), $-\ord_O(f) =2(d_2-d_3)+3>0$, or, equivalently, $3+2(d_2-d_1)>0$. Then in the first representation of $f(P)$ in (\ref{rep1}) we find that $B$ is coprime to the numerator and denominator of the second factor. Assume that $v \notin S'$ with $v(x(P))<0$, i.e., $v(B)>0$. Then $$v(f(P)) = -(3+2(d_2-d_3))v(B)<0,$$ and from $P \in \mathcal{P}(E,K,S,f)_n$ we conclude that $n \mid v(f(P))$, i.e., $$n \mid v(B) \cdot (3+2(d_2-d_3)).$$ The hypothesis $n \nmid \ord_O(f)$ implies that $n \mid v(B)$, i.e., $P \in \mathcal{P}(E',K,S',x)_{2n}$ with $n>1$ (i.e., $2n \nmid \ord_O(x)=-2$). 

Secondly, if in  (\ref{ordo}), $-\ord_O(f) =2(d_1-d_3)>0$, or, equivalently, $2(d_1-d_2)-3>0$, then  in the second representation of $f(P)$ in (\ref{rep2}) we find that $B$ is coprime to the numerator and denominator of the second factor. Assume that $v \notin S'$ with $v(x(P))<0$, i.e., $v(B)>0$. Then $$v(f(P)) = -2(d_1-d_3))v(B)<0,$$ and from $P \in \mathcal{P}(E,K,S,f)_n$ we conclude that $n \mid v(f(P))$, i.e., $$n \mid v(B) \cdot 2(d_1-d_3).$$ The hypothesis $n \nmid \ord_O(f)$ implies that $n \mid v(B)$, i.e., $P \in \mathcal{P}(E',K,S',x)_{2n}$ with $n>1$ (i.e., $2n \nmid \ord_O(x)=-2$). 
\end{proof}

\section{Bounding the exponent} 

Without loss of generality, we assume that $E$ is given by a Weierstrass equation in short form $y^2=x^3+ax+b$, with $j_E \notin K^p$ and $f=x$. For the next reduction, we take our inspiration from Bennett and Dahmen (\cite{BenDah}, Section 2) in using a classical syzygy for binary cubic forms, applied to the 2-division polynomial.  

\begin{proposition} \label{red3}
If $P \in \mathcal{P}(E,K,S,x)_n \neq \emptyset$, then, up to replacing $K$ by a sufficiently large separable extension and enlarging $S$ so that $\mathcal{O}_{K,S}$ is a PID, there exists a solution to
$$  X^3+Y^2=Z^{4n}   $$
with $X,Y,Z \in \cO_{K,S}$ pairwise coprime and $\nu(Z)=0$ for all $\nu \in S$, with $B_P=Z^n v$ for some $S$-unit $v$, where $x(P)=A_P/B_P^2$ is a representation in coprime $S$-integers.  
\end{proposition}

\begin{proof} There exists a finite separable extension $K'$ of $K$ such that $E(K) \subseteq 2E(K')$: it suffices to let $K'$ contain the coordinates of the solutions $Q$ to the equations $P=2Q$ for $P$ running through a finite set of generators for $E(K)$ (this can also be done without halving generators, see Remark \ref{isogeny} below). Separability of $K'/K$ follows from the fact that the degree of $K'/K$ is only divisible by powers of $2$ and $3$, and we assume $p \geq 5$. 

Replace $K$ by $K'$.  Without loss of generality, enlarge $S$ so that it contains all divisors of the discriminant $\Delta_E$ of $E$, and such that the coefficients of the Weierstrass model of $E$ are in  $\mathcal{O}_{K,S}$ and $\mathcal{O}_{K,S}$ is a principal ideal domain. Suppose that $P \in \mathcal{P}(E,K,S,x)$, and write  $2Q=P$ with $Q \in E(K)$, where $x(Q)=A_Q/B_Q^2$ with 
$A_Q, B_Q$ coprime in $\mathcal{O}_{K,S}$. Then
\[
\frac{A_P}{B_P^2}=\frac{B_Q^8\vartheta_2(A_Q/B_Q^2)}{B_Q^2\psi_2^2(A_Q/B_Q^2)B_Q^6}.
\]
where $$\vartheta_2(x) = x^4-2ax^2-8bx+a^2 \mbox{ and } \psi_2^2(x) = 4(x^3+ax+b)$$ are classical division polynomials. This gives a representation of $x(Q)$ in which numerator and denominator are in $\cO_{K,S}$, and (cf.\ e.g., Ayad \cite{MR1185022}) the greatest common divisors of numerator and denominator divides the discriminant $\Delta_E$ of $E$. Furthermore, the factors $B_Q^2$ and $\psi_2^2(A_Q/B_Q^2)B_Q^6$ are coprime. 

Consider the binary cubic form $$K_2(X,Y)=4(X^3+aXY^2+bY^3).$$ A classical result, a ``syzygy for the covariants'', apparently first discovered by Eisenstein \cite{Eisenstein} (cf.\ \cite{HM}), says the following: 
\begin{lemma}
If $F$ is a binary cubic form with discriminant $\Delta_F$, set $$H(x,y) = \frac{1}{4} \det \left( \begin{array}{cc} \frac{\partial^2 F}{\partial x \partial x} &  \frac{\partial^2 F}{\partial x \partial y} \\  \frac{\partial^2 F}{\partial x \partial y} &  \frac{\partial^2 F}{\partial y \partial y}  \end{array} \right) \mbox{ and } G(x,y) = \det \left( \begin{array}{cc} \frac{\partial F}{\partial x} &  \frac{\partial F}{\partial y} \\  \frac{\partial H}{\partial x} &  \frac{\partial H}{\partial y}  \end{array} \right).$$ Then
\begin{equation} \label{sy}  G^2 + 4H^3 =  - 27 \Delta_F F^2. \qed \end{equation}  
\end{lemma}

Now if $P \in \mathcal{P}(E,K,S,x)_n$ for some $n>2$, then $B_P=uC^n$ where $\nu(u)=0$ for $\nu \notin S$. We see that  
\[
K_2(A_Q,B_Q^2)=\psi_2^2(A_Q/B_Q^2)B_Q^6=u^2C^{2n}/\delta,   
\]
with $\nu(\delta) \neq 0$ only for the finitely many valuations $\nu$ for which $\nu(\Delta_E) \neq 0$, which are included in $S$. 
 
The syzygy (\ref{sy}) for $F=K_2$ (with $\Delta_F=\Delta_E$) gives an equation of the form $$aX^3+bY^2=Z^{4n},$$ where $X,Y,Z \in  \mathcal{O}_{K,S}$ are non-zero and $a,b$ are $S$-units with 
$$ a=-\frac{\delta}{27u^4\Delta_E},\ b=-\frac{4\delta}{27u^4\Delta_E},\ X=G(A_Q,B^2_Q),\ Y=H(A_Q,B_Q^2) \mbox{ and } Z=C.$$

Since the resultant of any pair of $F,G$ and $H$ is a divisor of $\Delta_E^3$ (as can be seen by direct computation, or as in Prop.\ 2.1 in \cite{BenDah}), we find that the only common divisors of any pair of $X,Y$ and $Z$ belongs to $S$, i.e., $X, Y$ and $Z$ are pairwise coprime $S$-integers. Furthermore, if $\nu(Z) \neq 0$ for some $\nu \in S$, fix a uniformizer  $\pi_{\nu}\in \cO_{K,S}$ for $\nu$ (this is possible since we assume $\cO_{K,S}$ is a PID), and replace the equation by $$a'X^3+b'Y^2=(Z')^{4n}$$ with $$a'=\pi_{\nu}^{-4n\nu(Z)} a, \  b'=\pi_{\nu}^{-4n\nu(Z)} b \mbox{  and }Z'=\pi_{\nu}^{-\nu(Z)} Z;$$ then the new equation has has $\nu(Z')=0$. Doing this for all such (finitely many) valuations, we may assume $\nu(Z)=0$ for $\nu \in S$. Note that $B_P=Z^n v$ for an $S$-unit $v$. 

Dirichlet's $S$-unit theorem for function fields (due to F.K.\ Schmidt, cf.\ e.g.\ \cite{Rosen}, 14.2) shows that there are only finitely many values of $a$ and $b$ up to sixth powers, so we can enlarge $K$ to contain the relevant sixth roots (separable since $p \geq 5$) to find a solution in $K$ to $X^3+Y^2=Z^{4n}$ with $X,Y,Z$ coprime  $S$-integers and $\nu(Z)=0$ for all $\nu \in S$. 
\end{proof} 

\begin{remark} \label{isogeny}
In explicit bounds, the following observation might be useful. The extension $K'/K$ such that $E(K) \subseteq 2E(K')$ that is needed at the start of the proof can be constructed independently of choosing generators for $E(K)$: if $P=(x(P),y(P))$ satisfies the Weierstrass equation, we find that 
$$ \prod_{T \in E[2]-O} (x(P)-x(T)) = y(P)^2 $$ is a square. Extend $S$ so that $\mathcal{O}_{K,S}$ is a PID. Now the common divisors of the factors on the left hand side divide $\Delta_E$. Therefore, if we extend $K$ to $K'$ so that all prime divisors of $\Delta_E$ and all elements of $\cO^*_{K,S}$ (in which squares have finite index by Dirichlet's unit theorem) become squares in $K'$, then all $x(P)-x(T)$ are squares in $K'$. Now the explicit formula for the 2-isogeny $[2] \colon E \mapsto E$ implies that $E(K) \subseteq 2E(K')$. 
\end{remark}

\begin{se}
 The (logarithmic) height of $x \in K$ is defined by 
$$ h(x) = - \sum_{\substack{\nu \in M_K \\ \nu(x)<0}} \nu(x). $$ 
Note that $h(x) \ge 0$  and $h(x) \in \Z$ for all $x \in K$. 
Let 
\[
h(x)_0 = \sum_{\substack{\nu \in M_K \\ \nu(x)>0}} \nu(x).
\]
Note that $\nu(1/x)=-\nu(x)$, so $\nu(x)>0$ if and only if $\nu(1/x)<0$. Thus, by the product formula, 
\begin{lemma} \label{l1}
For all $x \in K$, $h(1/x)=h(x)_0=h(x)$. \qed
\end{lemma}
\end{se}

We will apply the following theorem of Mason's (the ``abc-conjecture for function fields''): 
\begin{theorem}[Mason \cite{Mason}] \label{mason} Suppose that $\gamma_1, \gamma_2$ and $\gamma_3$ are non-zero elements of $K$ with 
$\gamma_1+\gamma_2+\gamma_3=0$ and $\nu(\gamma_1)=\nu(\gamma_2)=\nu(\gamma_3)$ for each valuation $\nu$ not in a finite set $T$. Then either $\gamma_1/\gamma_2 \in K^p$ or 
\[
h(\gamma_1/\gamma_2) \le |T| +2g_K-2. \qed
\] 
\end{theorem}

\begin{proposition} \label{3term} If $X,Y,Z \in \cO_{K,S}$ are pairwise coprime $S$-integers with $\nu(Z)=0$ for all $\nu \in S$, $Z \notin \F$  and $X^3/Z^N \notin K^p$, that satisfy an equation of the form 
$$  X^3+Y^2=Z^{N}   $$ for $N \geq 1$, then $N \leq C'$ for some constant $C'$ that depends on $K$ and $S$ only.
\end{proposition}

\begin{proof} 
Mason's theorem \ref{mason} applied to $\{\gamma_1,\gamma_2,\gamma_3\} = \{X^3/Z^N, Y^2/Z^N,1\}$ in all combinations, with $T=S \cup \{ \nu \colon \nu(X)>0 \mbox{ or } \nu(Y)>0 \mbox{ or } \nu(Z)>0\}$  implies: if $X^3/Y^2 \notin K^p$, then 
\begin{equation} \label{MT}
\max \{ h(X^3/Z^N), h(Y^2/Z^N), h(X^3/Y^2) \} \le 2g_K-2+|S|+h_0(XYZ),
\end{equation}
Using Lemma \ref{l1} and the fact that we are assuming  $\nu(Z)=0$ for all $\nu \in S$, we find
\begin{eqnarray*}
h(X^3/Z^N) &=& -\sum_{\substack{\nu \in S \\ \nu(X^3/Z^N)<0}} \nu(X^3/Z^N)+N\sum_{\substack{\nu \notin S \\ \nu(Z) >0}} \nu(Z)  \\
&=&-\sum_{\substack{\nu \in S \\ \nu(X^3)<0}} \nu(X^3)+Nh(Z) 
\end{eqnarray*}
and also
\begin{eqnarray*}
h(X^3/Z^N) &=& h(Z^N/X^3) \\
&=&  -\sum_{\substack{\nu \in S \\ \nu(Z^N/X^3)<0}} \nu(Z^N/X^3)+3\sum_{\substack{\nu \notin S \\ \nu(X)>0}} \nu(X) \\
&=& 3\sum_{\substack{\nu \in S \\ \nu(X)>0}} \nu(X) + 3\sum_{\substack{\nu \notin S \\ \nu(X)>0}} \nu(X) \\ &=& 3h(X).    
\end{eqnarray*}
Thus, $$h(X^3/Z^N)=3h(X) \ge Nh(Z).$$ Similarly,  we find $$h(Y^2/Z^N)=2h(Y) \ge Nh(Z).$$ We also have $$h(X^3/Y^2) \ge \max\{2h(Y), 3h(X)\}.$$ Combining this with the estimate (\ref{MT}) from Mason's Theorem yields  
\begin{equation} \label{MT2}
\max \{ 3h(X), 2h(Y), Nh(Z) \} \le 2g_K-2+|S|+h(X)+h(Y)+h(Z). 
\end{equation} 
Let $$\Sigma=\Sigma_{X,Y,Z}=h(X)+h(Y)+h(Z)$$ and $$C=C_{K,S}=2g_K-2+|S|.$$
From (\ref{MT2}), we find inequalities 
$$ h(X) \le \frac{1}{3} (\Sigma+C) \mbox{ and } h(Y) \le \frac{1}{2} (\Sigma+C) \mbox{ and } h(Z) \le \frac{1}{N} (\Sigma+C), $$
which add up to 
\[
\Sigma \le \left(\frac{1}{2} + \frac{1}{3} + \frac{1}{N}  \right) (\Sigma + C),
\]
or
\begin{equation} \label{esti}
\frac{1}{N} \ge \frac{1}{1+\frac{C}{\Sigma}}-\frac{5}{6}.
\end{equation}
Now there are two possibilities:
\begin{enumerate}
\item[Case 1.] $\Sigma > 11C$. From (\ref{esti}) it follows that $N < 12$.
\item[Case 2.] $\Sigma \le 11C$. Since $h(X), h(Y), h(Z) \in \Z$ are positive and bounded above by the constant $11C$ that depends only on $K$ and $S$, there must be finitely many choices for $X$, $Y$ and $Z$. Since $Z \notin \F$, $h(Z)>0$. Hence we find a bound on $N$, since $$N \leq Nh(Z) = h(Z^N) \leq \max\{ h(X^3+Y^2) \colon h(X)+h(Y) \leq 11C\}.$$ 
\end{enumerate}
\end{proof}

\begin{corollary} \label{f1} Assume $j_E \notin K^p$. 
There exists a constant $\tilde C$ only depending on $E,K$ and $S$ such that  $$\mathcal{P}(E,K,S,x) \subseteq \bigcup_{ n \leq \tilde C} \mathcal{P}(E,K,S,x)_n.$$ 
\end{corollary}

\begin{proof}
The successive enlargement of the original field $K$ and the original set of valuations $S$ only depended on $E,K$ and $S$. We assume we have extended the field and set so that we are in the situation of Proposition \ref{red3}. Let $P \in\mathcal{P}(E,K,S,x)_n \neq \emptyset$. Then in particular, $B_P$ is defined and in the notation of the two previous propositions, $Z^{n}=B_P v$ where $v$ is an $S$-unit. Propositions \ref{red3} and \ref{3term} with $N=4n$ imply that if $n>C'/4$ where $C'$ is the constant implied by Proposition \ref{3term}, then either of the following two cases occurs: 
\begin{enumerate}
\item $Z \in \F$; then $B_P$ is an $S$-unit and hence $P \in \mathcal{Q}(E,K,S,x) \subseteq \mathcal{P}(E,K,S,x)_p$;
\item $X^3/Z^{4n} \in K^p$; since $X$ and $Z$ are coprime $S$-integers, $Z^{4n}$ is a $p$-th power up to an $S$-unit; hence $B_P^4$ is a $p$-th power up to an $S$-unit, and hence (with $p$ odd) $B_P$ is a $p$-th power up to an $S$-unit, so $P \in \mathcal{P}(E,K,S,x)_p$.
\end{enumerate} 
Hence $$\mathcal{P}(E,K,S,x) \subseteq \bigcup_{ n \leq C'/4} \mathcal{P}(E,K,S,x)_n \cup  \mathcal{P}(E,K,S,x)_p,$$ so it suffices to take $\tilde C=C'/4+p$. 
\end{proof}

\begin{remark} \label{upp}
The proof also shows that if $j_E \notin K^p$ and $n$ is not divisible by $p$, then $$\mathcal{P}(E,K,S,f)_n = \mathcal{Q}(E,K,S,f)$$ for $n>C'/4$. 
\end{remark}

\section{Bounding the solutions}

By Corollary \ref{f1}, to prove the main theorem we are now reduced to showing the following: 

\begin{proposition} \label{fixedbound}
If $j_E \notin K^p$, then for fixed $n>2$, the set $\mathcal{P}(E,K,S,x)_n$ is finite. 
\end{proposition}

\begin{proof} The start of the proof is a function field version of the argument in \cite{Jthesis}, Theorem 5.2.1, which we then combine with the abc-hypothesis in function fields. This means we have to deal with the exceptional case  where a term is a $p$-th power, but we show that this implies that $j_E \in K^p$. 

Suppose that $P \in \mathcal{P}(E,K,S,x)_n$ for $n>2$. Without loss of generality, we assume $E$ is in short Weierstrass form with coefficients from $\cO_{K,S}$, and $K$ and $S$ have been extended so that $O_{K,S}$ is a PID, the $2$-torsion of $E$ is $K$-rational, and $\Delta_E$ is an $S$-unit. Let $\alpha_1, \alpha_2, \alpha_3$ denote the $x$-coordinates of the points in $E[2]$. Extend $S$ further so that the differences $\alpha_i-\alpha_j$ are $S$-units for $i \neq j$. The necessary field extension is separable, since $p \geq 5$. 

Let $P = (A_P/B_P^2,C_P/B_P^3)  \in E(K)$ where $A_PC_P$ and $B_P$ are coprime $S$-integers. Plugging the coordinates of $P$ into the Weierstrass equation gives 
$$ C_P^2 = \prod_{i=1}^3 (A_P-\alpha_i B_P). $$
The factors $A_P-\alpha_i B_P^2$ are coprime $S$-integers. Indeed, if $\nu \notin S$ has $\nu(A_P-\alpha_i B_P^2)>0$ and $\nu(A_P-\alpha_j B_P^2)>0$, then $\nu((\alpha_i-\alpha_j)B_P^2)>0$, so $\nu(B_P)>0$, and hence from $\nu(A_P-\alpha_i B_P^2)>0$ it follows that also $\nu(A_P)>0$, a contradiction. Hence
\begin{equation} \label{ddd} 
A_P-\alpha_i B_P^2=z_i^2
\end{equation} 
for some $z_i \in K$, up to $S$-units. By extending $K$ such that all $S$-units from $K$ become squares (which can be done by a finite extension by Dirichlet's unit theorem)---and keeping all previous conditions satisfied---, we absorb the $S$-unit into $z_i$. Since the necessary field extension is of degree a power of $2$, it is separable for $p \geq 5$. Taking the difference of any two of the equations (\ref{ddd}) yields
\begin{equation} \label{bb}
(\alpha_j-\alpha_i)B_P^2=(z_i+z_j)(z_i-z_j).
\end{equation}
Now $z_i+z_j$ and $z_i-z_j$ are coprime, since if $\nu(z_i+z_j)>0$ and $\nu(z_i-z_j)>0$ for $\nu \notin S$, then $\nu(z_i)>0$. But also $\nu(B_P)>0$ from (\ref{bb}), and hence $\nu(A_P)>0$ from (\ref{ddd}), a contradiction since $A_P$ and $B_P$ are coprime. 

Write $B_P=uB^\ell$ with an $S$-unit $u$ for some $B \in \cO_{K,S}$ and $n \in \Z$ with $\ell>1$ and $n=2\ell$. Then $z_i+z_j$ and $z_i-z_j$ are $n$-th powers up to $S$-units. For convenient notational purposes, let $\Delta$ denote a fixed choice of a plus or minus symbol, and $\nabla$ the opposite sign. We will use without further mentioning that $-1 \in K^p$. We distinguish the following cases:
\begin{enumerate}
\item \emph{There exists a set of distinct indices $i,j,k$ for which $\frac{z_i \pm z_j}{z_i \Delta z_k} \notin K^p$ for both signs $\pm$.}\\
We have the following \emph{Siegel's identities}:
\[
\frac{z_i \pm z_j}{z_i-z_k} \mp \frac{z_j \pm z_k}{z_i-z_k}=1=\frac{z_i \pm z_j}{z_i+z_k} \mp \frac{z_j \mp z_k}{z_i+z_k}.
\]
In our situation, they become equations of the form 
\begin{equation} \label{xyz}
aX^{2\ell}+bY^{2\ell}=1,
\end{equation} 
 $a,b \in \cO_{K,S}^*$ are $S$-units. Using the function field version of Dirichlet's unit theorem, there is a finite set $R$ of representatives for such units up to $2\ell$-th powers.  The above reasoning implies that $\frac{z_i \pm z_j}{z_i\Delta z_k}$ takes on values inside
 $$\mathcal{S}:= \{ a_0 X_0^{2\ell} \colon  X_0 \in K, a_0 \in R \mbox{ and } \exists Y_0 \in K, b_0 \in R \colon a_0 X_0^{2\ell} + b_0 Y_0^{2\ell} = 1 \}. $$
Mason's theorem implies that for $n>2$ (i.e., $\ell>1$), the solution set to such any of the finitely many ternary equations that occur in the definition of $\mathcal S$ is finite, since $\frac{z_i \pm z_j}{z_i \Delta z_k} = a X_0^{2\ell} \notin K^p$ by assumption.  

This implies that the set of values taken by \begin{equation} \label{z} Z_{\Delta} =\frac{1}{\alpha_j-\alpha_i} \cdot  \frac{z_i - z_j}{z_i \Delta z_k} \cdot  \frac{z_i + z_j}{z_i \Delta z_k} \end{equation} is finite. To finish the proof that $P$ takes on only finitely many values in this case, we state the following identity, which can be verified by direct computation, or follows from combining the last four indented formulas in the proof of 5.2.1 in \cite{Jthesis}: 
\begin{equation} \label{x} 4x(P) = 2(\alpha_i+\alpha_k) + Z_{\Delta}^{-1}+(\alpha_i-\alpha_k)^2 Z_{\Delta}, \end{equation}
and observe that to every value of $x(P)$ correspond at most two values of $P$. 
\item \emph{There exists a set of distinct indices $i,j,k$ for which $x_{\pm}:=\frac{z_i \pm z_j}{z_i \Delta z_k} \in K^p$ for both signs $\pm$.}\\
We claim that if this statement holds for one set of indices (for fixed $\Delta$), it holds for all sets of indices (for the same fixed $\Delta$). It suffices to prove it for the permuted indices $(j,i,k)$ and $(k,j,i)$, since these permutations generate $S_3$. The second permutation is implemented by replacing $x_{\pm}$ by $\pm(1-x_{\pm})$, which are $p$-th powers if and only if $x_{\pm}$ are so. The first is given by replacing $x_{\pm}$ by $-x_{\pm}/(1-x_{\pm})$. This proves the claim. 

We then conclude from the equalities
$$ \lambda := \frac{\alpha_1 - \alpha_2}{\alpha_1 - \alpha_3} = \frac{z_1-z_2}{z_1+z_3} \cdot \left(\frac{z_1-z_3}{z_1+z_2}\right)^{-1}  = \left(\frac{z_1+z_3}{z_1-z_2} \right)^{-1}\cdot \frac{z_1+z_2}{z_1-z_3} $$
(the first if $\Delta=+$ and the second if $\Delta=-$) 
that $\lambda$ is a $p$-th power. But now we have
\begin{equation*} j_E = 256 \frac{(\lambda^2-\lambda+1)^3}{\lambda^2(\lambda-1)^2}, \end{equation*}
(cf.\ \cite{SilvermanAEC}, III.1.7), and we conclude that $j_E \in K^p$, which we have assumed is not the case.
\item \emph{For all set of distinct indices $i,j,k$, we have $\frac{z_i \mp z_j}{z_i \Delta z_k} \in K^p$ and $\frac{z_i \pm z_j}{z_i \Delta z_k} \notin K^p$, for some choice of signs $\mp$ and $\pm$ (depending on the indices).}\\
We use the identity
$$ \frac{z_i \pm z_j}{z_i \Delta z_k} = \frac{   1 -  \frac{z_i \mp z_j}{z_i \Delta z_k}}{ 1 -  \frac{z_i \nabla z_k}{z_i \pm z_j}} $$
to see that $\frac{z_i \nabla z_k}{z_i \pm z_j} \notin K^p$. But together with the assumption that $\frac{z_i \pm z_j}{z_i \Delta z_k} \notin K^p$ (so also its inverse), this implies that both $\frac{z_i \nabla z_k}{z_i \pm z_j}$ and $\frac{z_i \Delta z_k}{z_i \pm z_j}$ are not $p$-th powers, and the first case applies. 
\end{enumerate}
This covers all cases and finishes the proof of the theorem. 
\end{proof} 

\begin{remark}
In particular, $\mathcal{P}(E,K,S,f)_p$ is finite, and hence there exists an integer $R$ such that $\mathcal{P}(E,K,S,f)_{p^k}=\mathcal{Q}(E,K,S,f)$ for all $k \geq R$. Then the bound in Remark \ref{upp} can be altered to: if $E$ is an elliptic curve with $j_E \notin K^p$, then $\mathcal{P}(E,K,S,f)_n-\mathcal{Q}(E,K,S,f)=\emptyset$ for $n \geq \max \{C'/4,p^R\}$. Since one can in principle use the above method to bound the height of elements in $\mathcal{P}(E,K,S,f)_p$, they can be listed and $R$ can be found. 
\end{remark}

\begin{remark}
It might seem that the above proof simultaneously bounds $\ell$ and the height of a solution, so that there is no need for proving Corollary \ref{f1} first. However, in general the maximal height of a set of representatives of $S$-units up to $2\ell$-th powers depends on $\ell$, making this reasoning impossible. In some special cases, e.g.\ when the field extension that is used has a \emph{finite} unit group, one can restrict the ``coefficients'' $a$ and $b$ to a finite set independent of $\ell$, and then such a simultaneous bound \emph{is} possible, see, e.g.\ Example \ref{eee} below. 
\end{remark}

\section{Explicit bounds} 

We now show some examples of explicit bounds. 

\begin{se} For this, we first list some crude estimates of heights in a rational function field $\F(u)$ (we write the variable as $u$ to avoid confusion when taking field extensions). The first two estimates are (see, e.g., \cite{BG} 1.5.14-15, but do the non-archimedean case): 
\begin{equation} \label{ha} \max\{h(xy),h(x+y)\} \leq h(x)+h(y) \mbox{ if } x,y \in \F(u); \end{equation} which implies the following: 
\begin{lemma} \label{haha}
If $X,Y \in \F_q(u)$ and $a,b \in \F_q[u]$, then $$h(aX^3+bY^2) \leq h(a)+h(b)+3h(X)+2h(Y).$$ In particular, if $h(X)+h(Y) \leq 11 C$ and $\max \{ h(a), h(b) \} \leq \kappa$, then 
$$ h(aX^3+bY^2) \leq 2\kappa+33C. \qed $$
\end{lemma}
For future use, we also list the following fact: \begin{equation} \label{h+}  h(\alpha x) \geq h(x)-h(\alpha) \mbox{ for all } \alpha \in \F[u] \mbox{ and } x \in \F(u). \end{equation} Indeed, writing $x=a/b$ with $\alpha=\alpha_0 \alpha_1$ and $b=\alpha_0 b_1$ with $b_1$ coprime to $\alpha$, we find $h(ax) = \max \{ \deg \alpha_1 +\deg a , \deg b_1 \} \geq \max \{\deg a , \deg b_1 + \deg \alpha_0 \} - \deg \alpha_0 \geq h(x)-h(\alpha).$ Using $h(x)=h(1/x)$, this implies that 
 \begin{equation} \label{h++}  h(\alpha x) \geq h(x)-2 h(\alpha) \mbox{ for all } \alpha \in \F(u) \mbox{ and } x \in \F(u). \end{equation}
\end{se} 

\begin{example} \label{expl} In this example, we make the strategy from the proof of Proposition \ref{red3} explicit. 
All calculations in this example were verified using MAGMA \cite{Magma}. Consider the curve \begin{equation} \label{exE} E \colon y^2 = x^3-t(t-2)x^2+2t^2(t+1)x \end{equation}
over $K=\F_5(t)$ of discriminant $\Delta_E=4t^6(t+1)^2(t-1)^2$ and $j$-invariant $j_E=-(t^2-2)^3/(t^2-1)^2$. Let $S=\{1/t\}$. The group $E(K)$ is the direct product of the full 2-torsion and a free group of rank one generated by $P=(t,t^2)$.  

We want to compute an explicit upper bound on the constant $C'$ from Remark \ref{upp}; in terms of the elliptic divisibility sequence associate to $P$ with values in $\F_5[t]$, this means that the maximal pure $\ell$-th power, non-constant and  with $\ell$ coprime-to-$5$, that occurs in it has $\ell \leq n/2 \leq C'/8$. The sequence starts
$$ 1,1,t^2-1, t^2+1, (t^3+t^2-2t-1)(t^3-t^2-2t+1), \dots . $$

For this, we look at the proof of Proposition \ref{red3}. At the start of the proof, the field $K$ is extended to contain the $x$-coordinates of a point $Q$ such that $P=2Q$. The point $Q$ exists over the field $K'=\F_5(T)$ where $T^2=t$ (actually, $x(Q)=T^2(T-2)$). Then we enlarge $S$ to contain the divisors of $\Delta_E$, so $S'=\{1/T,T,T\pm 1, T\pm 2\}$. The $S'$-class number is 1, so $\mathcal{O}_{K',S'}$ is already a PID. At the end of the proof, it seems computationally advantageous not to extend $K'$ further to contain $6$-th roots of $S$-units, but rather, to choose a set of representatives $R=\{2,T\pm 1, T\pm 2\}$ for $\mathcal{O}^*_{K',S'}/(\mathcal{O}^{\ast}_{K',S'})^6$ and rephrase the result from Corollary \ref{f1} explicitly (with $g_{K'}=0$ and $|S'|=6$) as \begin{equation} \label{newest} n \leq \min \{ 2, \frac{1}{4} \max_{a,b \in R} \{ h(aX^3+bY^2) \colon h(X)+h(Y) \leq 11 C  \} \},\end{equation} 
with $C = -2 + 6 = 4$. Observe that $\max \{ h(a) \colon a \in R \} = 1 =:\kappa$, and use Lemma \ref{haha} to find \begin{align*} h(aX^3+bY^2) \leq 2+33 \cdot 4 \end{align*} so that finally $n \leq 33$ and $\ell \leq 16$. 
\end{example}

\begin{examplepf} \label{eee} In this example, we show how, in some cases, the proof of Proposition \ref{fixedbound} can be changed so it implies a simultaneous bound on the exponent and the height of a perfect power, leading to a proof of Proposition \ref{introprop} from the introduction. 

Assume that $E$ is an elliptic curve over a rational function field $K=\F_q(t)$ with with $j_E \notin K^p$ and coefficients from $\F_q[t]$ (the ring of $S$-integers for $S=\{1/t\}$) such that all 2-torsion points on $E$ are $K$-rational, and assume that $P=2Q \in 2E(K)$ with associated elliptic divisibility sequence $\{B_n\}$. Set $S=\{1/t\}$, so $\mathcal{O}_{K,S} = \F_q[t]$. Suppose that $B_n = C^{\ell}$ for some $C \in \F_q[t]$. Since $P=2Q$, in the proof of Proposition \ref{fixedbound}, the expressions $A_P-\alpha_i B_P=z_i^2$ are actual squares in $\F_q[t]$, so that $(z_i-z_j)/(z_i \Delta z_k) = aX^{2 \ell}$ satisfies $a X^{2 \ell} + b Y^{2 \ell} = 1$ for some $a, b \in \F_q(t)$ whose numerator and denominator divide some of the $\alpha_i - \alpha_j$. In particular, they divide $\Delta_E$, so $\max \{h(a), h(b)\} \leq \deg \Delta_E$. If $x \in K$, let $n_0(x)$ denote the number of valuations $\nu$ for which $\nu(x) \neq 0$. The abc-hypothesis (Mason's theorem) implies a bound on the height of a possible solution $X$, as follows: 
\begin{align} \label{vbn} \max\{ h(aX^{2\ell}), h(bY^{2\ell})\} & \leq -2 + \# \{ \nu \colon \nu(aX^{2\ell}) \neq 0 \mbox{ or } \nu(bY^{2\ell}) \neq 0\} \\
&\leq -2 +  n_0(a) + n_0(b) + h(X) + 2 h(Y);   \end{align}
where we may write $h(X)+2h(Y)$ instead of $2h(X)+2h(Y)$ since the equation satisfied by $X$ and $Y$ implies that if $\nu(a)=\nu(b)=0$ and $\nu(X)<0$, then also $\nu(Y)<0$. 
Using (\ref{h++}), we find 
$$ \max\{ -2h(a) + 2 \ell h(X), -2h(b)+2 \ell h(Y) \} \leq -2 + n_0(a) + n_0(b) + h(X) + 2 h(Y), $$ 
which add up to give
\begin{align} \label{ytr} (\ell-1)h(X) & \leq (\ell-1) h(X) + (\ell-2) h(Y) \\ & \leq -2  + n_0(a) + n_0(b) + h(a)+h(b) \\ &\leq 2 (\deg \Delta_E + n_0(\Delta_E)-1) \\ & \leq 4 \deg \Delta_E,\end{align}
since $\max \{n_0(a), n_0(b) \} \leq n_0(\Delta_E) \leq \deg \Delta_E + 1$ (counting the valuation $\deg_t$). 
This implies in particular that $$\ell \leq 2 (\deg \Delta_E + n_0(\Delta_E))-1 \leq 4 \deg \Delta_E+1.$$ For symmetry reasons, the estimate (\ref{ytr}) also holds with $X$ replaced by $Y$. From (\ref{vbn}), we then find (with $\ell \geq 2$) that 
\begin{align*} h(aX^{2l}) &\leq  -2 + 2n_0(\Delta_E) + 3 h(X) \\ &\leq 6 \deg \Delta_E + 8 (n_0(\Delta_E)-1).\end{align*} With our previous estimates for height of sums and products, we deduce from this with (\ref{z}) and (\ref{x}) that 
$$ h(Z_\Delta) \leq  13 \deg \Delta_E + 16 (n_0(\Delta_E)-1) $$
and finally
$$ h(x(P)) \leq 29 \deg \Delta_E + 32 (n_0(\Delta_E)-1) \leq 61 \deg \Delta_E. $$ 
An estimate for the difference between the height and the canonical height can be deduced from the local (non-archimedean) counterparts (as in section 4 of \cite{SilvermanH}), and gives 
$$ -\frac{1}{24} h(j_E) \leq \hat h(P) - \frac{1}{2} h(x(P)) \leq 0, $$ 
so  we find \begin{align*} n  = \frac{\hat h (nP)}{\hat h(P)} & \leq \frac{h(x(nP))}{2(\frac{h(P)}{2}-\frac{h(j_E)}{24})} \\ &\leq \frac{29 \deg \Delta_E + 32 (n_0(\Delta_E)-1)}{h(P)-\frac{h(j_E)}{12}} \\ & \leq \frac{732 \deg \Delta_E}{12h(P)-h(j_E)}.\end{align*} 
Translated to the corresponding elliptic divisibility sequence, this proves Proposition \ref{introprop}. \qed
\end{examplepf}

\begin{example} \label{explH}
In Example \ref{expl}, $P=2Q$ over $K'=\F_5(T)$ with $T=t^2$. Thus, the method from Example \ref{eee} applies, but it turns out to produce a worse bound. However, in this concrete case one can improve the estimate as follows: instead of bounding $h(a)$ by $\deg \Delta_E$, we observe that the set of differences $\alpha_i-\alpha_j$ belongs to $\{ 2T^2, T^2(T^2+1), T^2(T^2-1)\}$, whose elements have maximal degree $4$ and maximal number of distinct prime divisors $3$, so $\max \{ \deg(a),\deg(b),n_0(a),n_0(b) \} \leq 4$. Thus, the estimates from the previous example can be improved to give
$ (\ell-1) h(X) \leq -2  + n_0(a) + n_0(b) + h(a)+h(b) \leq 14$, so $\ell \leq 15$. This result has no assumption of $\ell$ being coprime to $5$. We also find via this method that $n \leq 212$ (note that the computations are using the degree in $T$ (height in the field $K'$), but we are interested in the degree in $t$ (height in the original field $K$), so have have halved the final result). Finally, we computed in SAGE \cite{sage} that $B_3, B_4$ and $B_5$ have only simple factors, and hence any $B_n$ with $n$ divisible by $3, 4$ or $5$ has a simple factor, by the remark at the end of \ref{formal}. On the other hand, we also computed in SAGE that the only non-squarefree values of $B_n$ for $6 \leq n \leq 212$ have $n$ divisible by $3, 4$ or $5$.  We conclude that the only perfect power denominators occur for $n=1$ and $n=2$, which corresponds to $B_1=B_2=1$. 
Below is the SAGE command that produces these results:
\begin{quote} {\footnotesize {\tt
\begin{verbatim}
F = FractionField(PolynomialRing(GF(5),'t'))
t = F.gen()
E = EllipticCurve([0,-t*(t-2),0,2*t^2*(t+1),0])
P = E(t,t^2)
n = 0
p = 0
for counter in [1..212]:
  n=n+1
  p=denominator((n*P)[0])
  if ((p.degree()==0) or (not(sqrt(p).is_squarefree()) 
  and not(3.divides(n) or 4.divides(n) or 5.divides(n)))): print n
print 'Bye'
\end{verbatim} }}
\end{quote}
\end{example}

\section{The case of number fields}

\begin{se} We briefly indicate how the (non-uniform) abc-hypothesis for number fields implies the number field analogue of our main theorem. 
The abc-hypothesis for a number field $K$ was formulated by Vojta (\cite{Vojta}, p.\ 84), and it is easily seen to be equivalent to the following. For any $x \in K^*$ define
$$ h(x):= - \sum_{v(x)<0} v\left(x\right) \log N(\p_v), $$ where $N$ is the absolute norm and the non-archimedean valuation $v$ corresponds to the prime $\p_v$. 

\begin{conjecture}[abc-conjecture for number fields]
For any $\varepsilon>0$, there exists a positive constant $C_\varepsilon$, such that for all $\gamma_1, \gamma_2, \gamma_3 \in K^*$ satisfying $\gamma_1 + \gamma_2 + \gamma_3=0$, with $v(\gamma_1)=v(\gamma_2)=v(\gamma_3)$ for all non-archimedean valuations outside a finite set $T$ of such valuations, we have
$$  h\left(\frac{\gamma_1}{\gamma_3}\right) < (1+\varepsilon)\left(\sum_{v \in T} \log N(\p_v) \right) + C_\varepsilon.$$
\end{conjecture}

\begin{theorem} \label{mainNF} Assume that the abc-hypothesis holds for any number field. 
Let $K$ be a number field and $S$ a finite set of places of $K$, including all archimedean ones. Suppose that $E$ is an elliptic curve over $K$. Let $f$ denote a function in $K(E)$ with a pole of order $-\ord_O(f)>0$ at the zero point $O=O_E$ of the group $E$.
Then the set 
$$\bigcup_{n \nmid \ord_O(f)} \{ P \in E(K) : n \mid \nu(f(P)), \textrm{ for all } \nu \notin S \textrm{ with } \nu(f(P)) <0 \} $$
is finite. \qed
\end{theorem}

The proof is literally the same as the one for function fields (discarding issues of prime characteristic, taking into account the archimedean places, and making some constants $S$-units), once the abc-hypothesis is reformulated as above. 
\end{se}

\begin{remark}
In a number field, the finiteness of the set of points $P \in E(K)$ for which $n$ divides $\nu(f(P))$ for all $\nu \notin S$ with $\nu(f(P)) <0$ for a \emph{fixed} power $n$ (analogue of Proposition \ref{fixedbound}), has already been proven unconditionally in \cite{Jthesis}, Theorem 5.2.1, relying on Faltings's Theorem \cite{Faltings}. Thus, in number fields, one only needs to solve the ternary equation from Proposition \ref{3term}. This equation can be solved in particular cases using the modularity of $\Q$-curves; see forthcoming work of Dahmen and the second author. 
\end{remark}

\bibliographystyle{amsplain}

\end{document}